\documentclass{my-aims}

\usepackage{amsmath}
\usepackage{paralist}
\usepackage{graphics} 
\usepackage{epsfig} 
\usepackage{graphicx}  
\usepackage{epstopdf}
\usepackage[colorlinks=true]{hyperref}
\usepackage{subfigure}

\def\a{{\alpha(\cdot,\cdot)}}
\def\b{{\beta(\cdot,\cdot)}}
\def\t{\tau}
\def\e{\epsilon}
\def\LI{{_aI_t^{\a}}}
\def\RI{{_tI_b^{\a}}}
\def\LDa{{_aD_t^{\a}}}
\def\LDb{{_aD_t^{\b}}}
\def\RDa{{_tD_b^{\a}}}
\def\RDb{{_tD_b^{\b}}}
\def\LC{{^C_aD_t^{\a}}}
\def\RCa{{^C_tD_b^{\a}}}
\def\RCb{{^C_tD_b^{\b}}}
\def\DC{{^CD_\gamma^{\a,\b}}}

\hypersetup{urlcolor=blue, citecolor=red}
\usepackage{hyperref}

    \def\ppages{X--XX}

\newtheorem{theorem}{Theorem}[section]

\newtheorem{example}{Example}
\theoremstyle{definition}
\newtheorem{definition}[theorem]{Definition}
\newtheorem{remark}{Remark}

\title[Fractional Herglotz variational problems]{Fractional Herglotz variational problems\\ 
of variable order}

\author[D. Tavares, R. Almeida and D. F. M. Torres]{}



\email{dtavares@ipleiria.pt}
\email{ricardo.almeida@ua.pt}
\email{delfim@ua.pt}

\thanks{This work is part of first author's Ph.D., which is carried out 
at the University of Aveiro under the Doctoral Programme
\emph{Mathematics and Applications} of Universities of Aveiro and Minho.
It was supported by Portuguese funds through CIDMA
and \emph{The Portuguese Foundation for Science and Technology} (FCT),
within project UID/MAT/04106/2013. Tavares was also supported
by FCT through the Ph.D. fellowship SFRH/BD/42557/2007.}

\thanks{$^*$Corresponding author: Dina Tavares (dtavares@ipleiria.pt)}

\begin{document}

\maketitle

\centerline{\scshape Dina Tavares$^*$}
\medskip
{\footnotesize
\centerline{ESECS, Polytechnic Institute of Leiria, 2411--901 Leiria, Portugal}
\centerline{and} 
\centerline{Center for Research and Development in Mathematics and Applications (CIDMA)}
\centerline{Department of Mathematics, University of Aveiro, 3810--193 Aveiro, Portugal}}

\medskip

\centerline{\scshape Ricardo Almeida and Delfim F. M. Torres}
\medskip
{\footnotesize
\centerline{Center for Research and Development in Mathematics and Applications (CIDMA)}
\centerline{Department of Mathematics, University of Aveiro, 3810--193 Aveiro, Portugal}}


\bigskip


\begin{abstract}
We study fractional variational problems 
of Herglotz type of variable order. 
Necessary optimality conditions, described
by fractional differential equations depending 
on a combined Caputo fractional derivative 
of variable order, are proved. Two different cases 
are considered: the fundamental problem, 
with one independent variable, and the general case, 
with several independent variables.
We end with some illustrative examples
of the results of the paper.
\end{abstract}


\section{Introduction}

The theory of fractional calculus is an extension of ordinary calculus 
that considers integrals and derivatives of arbitrary real 
or complex order. Although its birth goes back to Euler, 
fractional calculus has gained a great importance 
only in recent decades, with the applicability 
of such operators for the efficient dynamic modeling 
of some real phenomena \cite{Coimbra,Sun}. More recently, 
a general theory of fractional calculus was presented, 
where the order of the fractional operators is not constant 
in time \cite{Samko_2}. This is a natural extension, 
since fractional derivatives are nonlocal operators 
and contain memory. Therefore, it is reasonable that 
the order of the derivative may vary along time.

The variational problem of Herglotz is a generalization 
of the classical variational problem \cite{MR3462534,MyID:342}. 
It allows us to describe nonconservative processes, 
even in the case when the Lagrangian is autonomous 
(that is, when the Lagrangian does not depend explicitly on time). 
In contrast to the calculus of variations, where the cost functional 
is given by an integral depending only on time, space and velocity, 
in the Herglotz variational problem the model is given by a differential 
equation involving the derivative of the objective functional $z$
and the Lagrange function depends on time, trajectories $x$ and $z$ 
and on the derivative of $x$. The problem of Herglotz was posed 
by Herglotz himself in 1930 \cite{Herglotz}, but only in 1996, 
with the works \cite{Guenther,Guenther:book}, it has gained 
a wide attention from the mathematical community. Indeed, 
since 1996, several papers were devoted to this subject: see
\cite{Almeida,Georgieva,Georgieva:sev,Santos:Viet,Santos:Disc,Santos:Spri,MR3462534,MyID:342}
and references therein.


\section{Preliminaries}

In this section we present some needed concepts and results.

\subsection{The fractional calculus of variable order}
\label{sec:FC}

We deal with fractional operators of variable fractional order 
on two variables, with range on the open interval $(0,1)$, 
that is, the order is a function $\alpha:[a,b]^2\to(0,1)$.
Given a function $x:[a,b]\to\mathbb{R}$, we present two 
different concepts of fractional derivatives of $x$.
First, we recall the definition of fractional 
integral \cite{Tatiana:IDOTA2011}.

\begin{definition}
The left Riemann--Liouville fractional integral 
of order $\a$ of $x$ is defined by
$$
\LI x(t)=\int_a^t \frac{1}{\Gamma(\alpha(t,\t))}(t-\t)^{\alpha(t,\t)-1}x(\t)d\t
$$
and the right Riemann--Liouville fractional integral of $x$ by
$$
\RI x(t)=\int_t^b\frac{1}{\Gamma(\alpha(\t,t))}(\t-t)^{\alpha(\t,t)-1} x(\t)d\t.
$$
\end{definition}

For fractional derivatives, we consider two types: 
Riemann--Liouville and Caputo fractional derivatives.

\begin{definition} 
The left Riemann--Liouville fractional derivative 
of order $\a$ of $x$ is defined by
$$
\LDa x(t)=\frac{d}{dt}\int_a^t
\frac{1}{\Gamma(1-\alpha(t,\t))}(t-\t)^{-\alpha(t,\t)}x(\t)d\t
$$
and the right Riemann--Liouville fractional derivative of $x$ by
$$
\RDa x(t)=\frac{d}{dt}\int_t^b
\frac{-1}{\Gamma(1-\alpha(\t,t))}(\t-t)^{-\alpha(\t,t)} x(\t)d\t.
$$
\end{definition}

\begin{definition} 
The left Caputo fractional derivative of order $\a$ of $x$ is defined by
$$
\LC x(t)=\int_a^t
\frac{1}{\Gamma(1-\alpha(t,\t))}(t-\t)^{-\alpha(t,\t)}x^{(1)}(\t)d\t
$$
and the right Caputo fractional derivative of $x$ by
$$
\RCa x(t)=\int_t^b
\frac{-1}{\Gamma(1-\alpha(\t,t))}(\t-t)^{-\alpha(\t,t)}x^{(1)}(\t)d\t.
$$
\end{definition}

Motivated by the works \cite{Malin:Tor,MyID:207}, we consider here
a generalization of previous concepts by introducing a linear 
combination of the fractional derivatives of variable fractional order.

\begin{definition}
\label{def1}
Let $\alpha, \beta: [a,b]^2\rightarrow(0,1)$ 
be two functions and $\gamma=(\gamma_1,\gamma_2) \in [0,1]^{2}$ a vector.
The combined Riemann--Liouville fractional derivative 
of function $x$ is defined by
\begin{equation}
\label{eq:cfd:RL}
D_\gamma^{\a,\b}x(t)=\gamma_1 \, \LDa x(t)+\gamma_2 \, \RDb x(t).
\end{equation}
Similarly, the combined Caputo fractional derivative 
of function $x$ is defined by
\begin{equation}
\label{eq:cfd:C}
^{C}D_\gamma^{\a,\b}x(t)=\gamma_1 \, \LC x(t)+\gamma_2 \, \RCb x(t).
\end{equation}
\end{definition}

When dealing with variational problems and necessary optimality
conditions, an important ingredient is always an integration 
by parts formula. Here, we present two such formulas, 
involving the Caputo fractional derivative of variable order.

\begin{theorem}[See \protect{\cite[Theorem 3.2]{Od}}]
\label{thm:FIP}
If $x,y \in C^1[a,b]$, then
$$
\int_{a}^{b}y(t) \, \LC x(t)dt
=\int_a^b x(t) \, {\RDa}y(t)dt+\left[x(t) 
\, {_tI_b^{1-\a}}y(t) \right]_{t=a}^{t=b}
$$
and
$$
\int_{a}^{b}y(t) \, {\RCa}x(t)dt=\int_a^b x(t) 
\, {\LDa} y(t)dt-\left[x(t) \, {_aI_t^{1-\a}}y(t)\right]_{t=a}^{t=b}.
$$
\end{theorem}

To end this short introduction to
the fractional calculus of variable order, we introduce 
one more notation. The dual fractional derivative 
of \eqref{eq:cfd:RL} is defined by
$$
D_{\overline{\gamma}}^{\b,\a}=\gamma_2 \, {_aD_t^{\b}}
+\gamma_1 \, {_tD_T^{\a}},
$$
where $\overline{\gamma}=(\gamma_2,\gamma_1)$
and $T\in[a,b]$ is the final time of the problem 
under consideration (see \eqref{funct1} below). 
The dual fractional derivative of 
\eqref{eq:cfd:C} is defined similarly.


\subsection{The fractional calculus of variations}

Let $D$ denote the linear subspace of $C^1([a,b])\times [a,b]$ defined by
\begin{equation*}
D=\left\{  (x,t)\in C^1([a,b])\times [a,b] : \DC x(t) 
\mbox{ exists and is continuous on }[a,b] \right\}.
\end{equation*}
We endow $D$ with the following norm:
$$
\|(x,t)\|=\max_{a\leq t \leq  b}|x(t)|
+\max_{a\leq t \leq b}\left| \DC x(t)\right|+|t|.
$$
To fix notation, throughout the text we denote by 
$\partial_i \psi$ the partial derivative of a function 
$\psi:\mathbb{R}^{n} \rightarrow\mathbb{R}$ with respect 
to its $i$th argument, $i = 1, \ldots, n$. For simplicity of notation, 
we also introduce the operator $[\cdot]_\gamma^{\alpha, \beta}$ defined by
$$
[x]_\gamma^{\alpha, \beta}(t)=\left(t, x(t), \DC x(t)\right).
$$
Let $L$ be a Lagrangian 
$L:C^{1}\left([a,b]\times \mathbb{R}^2 \right)\to\mathbb{R}$. 
Consider the following problem of the calculus of variations:
minimize functional $\mathcal{J}:D\rightarrow \mathbb{R}$ with
\begin{equation}
\label{funct1}
\mathcal{J}(x,T)=\int_a^T 
L[x]_\gamma^{\alpha, \beta}(t) dt + \phi(T,x(T))
\end{equation}
over all $(x,T)\in D$ satisfying the initial condition $x(a)=x_a$,
for a given $x_a\in \mathbb{R}$. The terminal time $T$ 
and the terminal state $x(T)$ are considered here free. 
The \emph{terminal cost function}  
$\phi:[a,b]\times \mathbb{R}\to\mathbb{R}$ is at least of class $C^1$.

\begin{theorem}[See \cite{Tavares}]
\label{teo1}
If $(x,T)$ is a minimizer of functional \eqref{funct1} on $D$, 
then $(x,T)$ satisfies the fractional differential equations
\begin{equation}
\label{ELeq_1}
\partial_2 L[x]_\gamma^{\alpha, \beta}(t)
+D{_{\overline{\gamma}}^{\b,\a}}\partial_3 L[x]_\gamma^{\alpha, \beta}(t)=0
\end{equation}
on $[a,T]$ and
\begin{equation}
\label{ELeq_2}
\gamma_2\left({\LDb}\partial_3 L[x]_\gamma^{\alpha, \beta}(t)
-{ _TD{_t^{\b}}\partial_3 L[x]_\gamma^{\alpha, \beta}(t)}\right)=0
\end{equation}
on $[T,b]$. Moreover, the following transversality conditions hold:
\begin{equation}
\label{CT1}
\begin{cases}
L[x]_\gamma^{\alpha, \beta}(T)+\partial_1\phi(T,x(T))+\partial_2\phi(T,x(T))x'(T)=0,\\
\left[\gamma_1 \, {_tI_T^{1-\a}} \partial_3L[x]_\gamma^{\alpha, \beta}(t)
-\gamma_2 \, {_TI_t^{1-\b}} \partial_3 L[x]_\gamma^{\alpha, \beta}(t)\right]_{t=T}
+\partial_2 \phi(T,x(T))=0,\\
\gamma_2 \left[ {_TI_t^{1-\b}}\partial_3 L[x]_\gamma^{\alpha, \beta}(t)
-{_aI_t^{1-\b}\partial_3L[x]_\gamma^{\alpha, \beta}(t)}\right]_{t=b}=0.
\end{cases}
\end{equation}
\end{theorem}

We can rewrite the transversality conditions 
\eqref{CT1}, obtaining the next result.

\begin{theorem}[See \cite{Tavares}]
\label{teo2}
If $(x,T)$ is a minimizer of functional \eqref{funct1} on $D$, 
then the fractional Euler--Lagrange equations \eqref{ELeq_1}
and \eqref{ELeq_2} are satisfied together with
the following transversality conditions:
\begin{equation*}
\begin{cases}
L[x]_\gamma^{\alpha, \beta}(T)+\partial_1\phi(T,x(T))\\
\qquad + x'(T) \left[ \gamma_2 {_TI}_t^{1-\b} \partial_3L[x]_\gamma^{\alpha, \beta}(t)
- \gamma_1 {_tI_T^{1-\a} \partial_3L[x]_\gamma^{\alpha, \beta}(t)} \right]_{t=T} =0,\\
\left[ \gamma_1\, {_tI_T^{1-\a}} \partial_3L[x]_\gamma^{\alpha, \beta}(t)
- \gamma_2\, {_TI_t^{1-\b}} \partial_3L[x]_\gamma^{\alpha, \beta}(t)\right]_{t=T}
+\partial_2 \phi(T,x(T))=0,\\
\gamma_2 \left[ _TI_t^{1-\b}\partial_3 L[x]_\gamma^{\alpha, \beta}(t)
-{_aI_t^{1-\b}\partial_3L[x]_\gamma^{\alpha, \beta}(t)}\right]_{t=b}=0.
\end{cases}
\end{equation*}
\end{theorem}


\section{Herglotz's variational principle}
\label{sec:theorems}

In this section we present a fractional variational principle 
of Herglotz depending on Caputo fractional derivatives. 

Let $\alpha,\beta: [a,b]^2\rightarrow (0,1)$ be two functions. 
The fractional Herglotz variational problem that we study 
consists in the determination of trajectories 
$x \in C^{1}\left([a,b]\right)$, satisfying a given initial 
condition $x(a)=x_{a} \in \mathbb{R}$, and a real $T \in [a,b]$ 
that extremize the value of $z(T)$, where $z$ satisfies 
the following differential equation with dependence 
on a combined Caputo fractional derivative operator:
\begin{equation}
\label{funct1H}
\dot{z}(t)=L\left(t,x(t), \DC x(t), z(t) \right),  
\quad  t \in [a,b],
\end{equation}
subject to the initial condition
\begin{equation}
\label{ICond}
z(a)=z_{a},
\end{equation}
where $z_a$ is a given real number.
In the sequel, we use the auxiliary notation
$$
[x,z]_\gamma^{\alpha, \beta}(t)=\left(t,x(t), \DC x(t), z(t) \right).
$$
The Lagrangian $L$ is assumed to satisfy the following hypothesis:
\begin{enumerate}
\item $L \in C^{1}\left([a,b] \times \mathbb{R}^{3},\mathbb{R}\right)$,

\item $t\rightarrow \lambda(t)\partial_3L[x,z]_\gamma^{\alpha, \beta}(t)$ 
is such that $_TD_t^{\b} \left( \lambda(t) 
\partial_3L[x,z]_\gamma^{\alpha, \beta}(t)\right)$,
\begin{equation*}
\LDb  \left( \lambda(t) \partial_{3} L[x,z]_\gamma^{\alpha, \beta}(t) \right)\ 
\text{ and } \ 
D_{\overline{\gamma}}^{\b,\a} 
\left( \lambda(t) \partial_{3} L[x,z]_\gamma^{\alpha, \beta}(t) \right)
\end{equation*}  
exist and are continuous on $[a,b]$, where
$$
\lambda(t)=\exp \left(-\int_a^t \partial_4 
L\left[x,z \right]_\gamma^{\alpha, \beta}(\tau)d\tau \right).
$$
\end{enumerate}

The following result gives necessary conditions 
of Euler--Lagrange type for an admissible 
function $x$ to be solution of the problem. 

\begin{theorem}
\label{mainteo}
Let $x \in C^{1}\left([a,b]\right)$ be such that $z$ defined 
by \eqref{funct1H} subject to the initial condition \eqref{ICond} 
has an extremum. Then, $(x,z)$ satisfies the fractional differential equations
\begin{equation}
\label{CEL1_Herg}
\partial_2 L[x,z]_\gamma^{\alpha, \beta}(t)\lambda(t)
+D{_{\overline{\gamma}}^{\b,\a}}\left(\lambda(t)
\partial_3 L[x,z]_\gamma^{\alpha, \beta}(t)\right)=0
\end{equation}
on $[a,T]$ and
\begin{equation}
\label{CEL2_Herg}
\gamma_2\left({\LDb} \left(\lambda(t)
\partial_3 L[x,z]_\gamma^{\alpha, \beta}(t)\right)
-{ _TD{_t^{\b}}\left(\lambda (t)\partial_3 
L[x,z]_\gamma^{\alpha, \beta}(t)\right)}\right)=0
\end{equation}
on $[T,b]$.
Moreover, the following transversality conditions are satisfied:
\begin{equation}
\label{CTransH}
\left\{
\begin{array}{l}
\left[\gamma_1 {_tI_T^{1-\a}} \left(\lambda (t) 
\partial_3L[x,z]_\gamma^{\alpha, \beta}(t)\right)
-\gamma_2 {_TI_t^{1-\b}} \left(\lambda (t) 
\partial_3L[x,z]_\gamma^{\alpha, \beta}(t)\right)\right]_{t=T}=0,\\
\gamma_2 \left[ {_TI_t^{1-\b}} \left( \lambda (t) 
\partial_3L[x,z]_\gamma^{\alpha, \beta}(t)\right) -{_aI_t^{1-\b} \left( \lambda (t) 
\partial_3L[x,z]_\gamma^{\alpha, \beta}(t)\right)}\right]_{t=b}=0.
\end{array}\right.
\end{equation}
If $T<b$, then $L[x,z]_\gamma^{\alpha, \beta}(T)=0$.
\end{theorem}

\begin{proof}
Let $x$ be a solution to the problem and consider an admissible variation 
of $x$, $\overline {x}= x+\e{h}$, where $h\in C^1([a,b])$ is an arbitrary 
perturbation curve and $\e \in\mathbb{R}$ represents a small number 
$\left(\e\rightarrow 0\right)$. The constraint $x(a)=x_a$ implies 
that all admissible variations must fulfill the condition $h(a)=0$.
On the other hand, consider an admissible variation of $z$, 
$\overline {z}= z+\e\theta$, where $\theta$ is a perturbation 
curve (not arbitrary) such that
\begin{enumerate}
\item $\theta(a)=0$, so that $z(a)=z_{a}$,

\item $\theta (T)=0$, because $z(T)$ is a maximum 
($\overline{z}(T)-z(T) \leq 0$) or a minimum 
($\overline{z}(T)-z(T)\geq0$),

\item $\theta (t)= \dfrac {d}{d \varepsilon} z(\overline {x},t) 
\biggr\rvert_{\varepsilon=0}$, so that the variation satisfies 
equation \eqref{funct1H}.
\end{enumerate}
Differentiating $\theta$ with respect to $t$, we obtain that
\begin{equation*}
\begin{split}
\dfrac{d}{dt}\theta (t)
& =\dfrac{d}{dt} \dfrac{d}{d\varepsilon} 
z(\overline {x},t) \biggr\rvert_{\varepsilon=0}\\
&= \dfrac{d}{d\varepsilon}\dfrac{d}{dt} 
z(\overline {x},t) \biggr\rvert_{\varepsilon=0}\\
&= \dfrac{d}{d\varepsilon} L\left(t,x(t)+\e{h(t)}, 
\DC x(t)+\e\DC {h(t)}, z(t) \right) \biggr\rvert_{\varepsilon=0}
\end{split}
\end{equation*}
and rewriting this relation, we obtain 
the following differential equation for $\theta$:
\begin{equation*}
\dot{\theta}(t) - \partial_{4}
L[x,z]_\gamma^{\alpha, \beta}(t) \theta(t) 
=\partial_{2} L[x,z]_\gamma^{\alpha, \beta}(t) h(t) 
+ \partial_{3}L[x,z]_\gamma^{\alpha, \beta}(t) \DC h(t).
\end{equation*}
Considering $\lambda(t)=\exp \left(- \displaystyle 
\int_a^{t} \partial_{4}L[x,z]_\gamma^{\alpha, \beta}(\tau)d\tau \right)$, 
we obtain the solution for the last differential equation:
\begin{equation*}
\theta(T)\lambda(T) - \theta(a) 
= \int_a^{T} \left(\partial_{2} L[x,z]_\gamma^{\alpha, \beta}(t) h(t) 
+ \partial_{3}L[x,z]_\gamma^{\alpha, \beta}(t) \DC h(t) \right) \lambda(t) dt.
\end{equation*}
By hypothesis, $\theta(a)=0$. If $x$ is such that $z(x,t)$ defined 
by \eqref{funct1H} attains an extremum at $t=T$, 
then $\theta(T)$ is identically zero. Hence, we get
\begin{equation}
\label{solutionPH}
\int_a^{T} \left(\partial_{2} L[x,z]_\gamma^{\alpha, \beta}(t)h(t) 
+ \partial_{3}L[x,z]_\gamma^{\alpha, \beta}(t) \DC h(t) \right) \lambda(t) dt = 0.
\end{equation}
Considering only the second term in \eqref{solutionPH}, 
and the definition of combined Caputo derivative, 
we obtain that
\begin{equation*}
\begin{split}
&\int_a^{T} \lambda(t) \partial_{3} L[x,z]_\gamma^{\alpha, \beta}(t) \left( \gamma_{1} 
\LC h(t) + \gamma_{2} \RCb h(t) \right)  dt\\
&=\gamma_{1} \int_a^{T} \lambda(t) \partial_{3} L[x,z]_\gamma^{\alpha, \beta}(t) \LC h(t)dt\\
&\ +\gamma_{2}  \left[ \int_a^{b} \lambda(t) \partial_{3} 
L[x,z]_\gamma^{\alpha, \beta}(t) \RCb h(t)dt - \int_T^{b} \lambda(t) 
\partial_{3} L[x,z]_\gamma^{\alpha, \beta}(t) \RCb h(t)dt \right]\\
&=\star.
\end{split}
\end{equation*}
Using Theorem~\ref{thm:FIP}, and considering 
$\overline {\gamma} =(\gamma_{2}, \gamma_{1})$, 
we get
\begin{equation*}
\begin{split}
\star
&= \int_a^{T} h(t) D_{\overline{\gamma}}^{\b,\a} 
\left( \lambda(t) \partial_{3} L[x,z]_\gamma^{\alpha, \beta}(t) \right) dt \\
&+ \int_T^{b} \gamma_{2} h(t) \left[ \LDb  \left( 
\lambda(t) \partial_{3} L[x,z]_\gamma^{\alpha, \beta}(t) \right)
- _TD_t^{\b} \left( \lambda(t) \partial_{3} L[x,z]_\gamma^{\alpha, \beta}(t) \right) \right] dt\\
& + h(T) \left[ \gamma_{1} {}_tI_{T}^{1-\a} \left( \lambda(t) 
\partial_{3} L[x,z]_\gamma^{\alpha, \beta}(t) \right) - \gamma_{2} {}_TI_{t}^{1-\b} 
\left( \lambda(t) \partial_{3} L[x,z]_\gamma^{\alpha, \beta}(t) \right) \right]_{t=T}\\
&+ h(b) \gamma_{2} \left[ _TI_{t}^{1-\b} \left( \lambda(t) 
\partial_{3} L[x,z]_\gamma^{\alpha, \beta}(t) \right) - _aI_{t}^{1-\b} \left( \lambda(t) 
\partial_{3} L[x,z]_\gamma^{\alpha, \beta}(t) \right) \right]_{t=b}.
\end{split}
\end{equation*}
Substituting this relation into expression \eqref{solutionPH}, we obtain
\begin{equation*}
\begin{split}
& \int_a^{T} h(t) \left[ \partial_{2} L[x,z]_\gamma^{\alpha, \beta}(t) \lambda(t) 
+ D_{\overline{\gamma}}^{\b,\a} \left( \lambda(t) 
\partial_{3} L[x,z]_\gamma^{\alpha, \beta}(t) \right)\right] dt \\
&+ \int_T^{b} \gamma_{2} h(t) \left[ \LDb  \left( \lambda(t) 
\partial_{3} L[x,z]_\gamma^{\alpha, \beta}(t) \right)  - _TD_t^{\b} \left( \lambda(t) 
\partial_{3} L[x,z]_\gamma^{\alpha, \beta}(t) \right) \right] dt\\
&+ h(T) \left[ \gamma_{1} {}_tI_{T}^{1-\a} \left( \lambda(t) 
\partial_{3} L[x,z]_\gamma^{\alpha, \beta}(t) \right) - \gamma_{2} {}_TI_{t}^{1-\b} 
\left( \lambda(t) \partial_{3} L[x,z]_\gamma^{\alpha, \beta}(t) \right) \right]_{t=T}\\
&+ h(b) \gamma_{2} \left[ _TI_{t}^{1-\b} \left( \lambda(t) 
\partial_{3} L[x,z]_\gamma^{\alpha, \beta}(t) \right) - _aI_{t}^{1-\b} \left( \lambda(t) 
\partial_{3} L[x,z]_\gamma^{\alpha, \beta}(t) \right) \right]_{t=b} =0.
\end{split}
\end{equation*}
With appropriate choices for the variations $h(\cdot)$, 
we get the Euler--Lagrange equations \eqref{CEL1_Herg}--\eqref{CEL2_Herg} 
and the transversality conditions \eqref{CTransH}.
\end{proof}

\begin{remark}
If $\a$ and $\b$ tend to 1, and if the Lagrangian $L$ 
is of class $C^2$, then the first Euler--Lagrange 
equation \eqref{CEL1_Herg} becomes
$$
\partial_2 L[x,z]_\gamma^{\alpha, \beta}(t)\lambda(t)+(\gamma_2- \gamma_1)
\frac{d}{dt}\left[\lambda(t)\partial_{3} L[x,z]_\gamma^{\alpha, \beta}(t)\right]=0.
$$
Differentiating and considering the derivative 
of the lambda function, we get
\begin{multline*}
\lambda(t) \Biggl[\partial_2 L[x,z]_\gamma^{\alpha, \beta}(t)\\
+(\gamma_2 - \gamma_1)\left[-\partial_{4} L[x,z]_\gamma^{\alpha, \beta}(t)
\partial_{3} L[x,z]_\gamma^{\alpha, \beta}(t)
+\frac{d}{dt}\partial_{3} L[x,z]_\gamma^{\alpha, \beta}(t)\right]\Biggr]=0.
\end{multline*}
As $\lambda(t)>0$ for all $t$, we deduce that
$$
\partial_2 L[x,z]_\gamma^{\alpha, \beta}(t)
+(\gamma_2- \gamma_1)\left[
\frac{d}{dt}\partial_{3} L[x,z]_\gamma^{\alpha, \beta}(t)
-\partial_{4} L[x,z]_\gamma^{\alpha, \beta}(t)
\partial_{3} L[x,z]_\gamma^{\alpha, \beta}(t)\right]=0.
$$
\end{remark}


\section{The case of several independent variables}

We now obtain a generalization of Herglotz's principle 
of Section~\ref{sec:theorems} for problems involving 
$n+1$ independent variables. Define 
$\Omega=\prod_{i=1}^{n} [a_i,b_i]$, with $n \in \mathbb{N}$, 
$P=[a,b]\times \Omega$ and consider the vector 
$s=(s_1, s_2, \ldots, s_n)\in \Omega$. The new problem consists 
in determining the trajectories $x \in C^{1}\left(P\right)$ 
that give an extremum to $z[x,T]$, where the functional 
$z$ satisfies the differential equation
\begin{multline}
\label{funct_siv1}
\dot{z}(t)=\int_{\Omega}L\left(t,s, x(t,s), \DC x(t,s),\right.\\
\left.^CD_{\gamma^1}^{\alpha_1(\cdot,\cdot),\beta_1(\cdot,\cdot)}x(t,s),\ldots, ^CD_{\gamma^n}^{\alpha_n(\cdot,\cdot),\beta_n(\cdot,\cdot)}x(t,s), z(t) \right)d^{n}s
\end{multline}
subject to the constraint
\begin{equation}
\label{herg:bound} 
x(t,s)=g(t,s)\quad \mbox{ for all }(t,s) \in \partial P,
\end{equation}
where $\partial P$ is the boundary of $P$ and $g$ is a given function 
$g:\partial P \rightarrow \mathbb{R}$. Assume that
\begin{enumerate}
\item $\alpha, \alpha_i, \beta, 
\beta_i: [a,b]^2 \rightarrow (0,1)$ with $i=1,\ldots n$,

\item $\gamma,\gamma^1, \ldots, \gamma^n \in [0,1]^2$,

\item $d^{n}s=ds_1\ldots ds_n$,

\item $\DC x(t,s)$, $^CD_{\gamma^1}^{\alpha_1(\cdot,\cdot),\beta_1(\cdot,\cdot)}x(t,s),
\ldots, ^CD_{\gamma^n}^{\alpha_n(\cdot,\cdot),\beta_n(\cdot,\cdot)}x(t,s)$ 
exist and are continuous functions,

\item the Lagrangian $L:P\times\mathbb{R}^{n+3} 
\rightarrow \mathbb{R}$ is of class $C^1$.
\end{enumerate}

\begin{remark} 
By $\DC x(t,s)$ we mean the Caputo fractional derivative 
with respect to the independent variable $t$ and by 
$^CD_{\gamma^i}^{\alpha_i(\cdot,\cdot),\beta_i(\cdot,\cdot)}x(t,s)$ 
we mean the Caputo fractional derivative with respect 
to the independent variable $s_i$, $i=1,\ldots,n$.
\end{remark}

In the sequel, we use the auxiliary notation
\begin{multline*}
[x,z]_{n, \gamma}^{\alpha, \beta}(t,s)
=\left(t,s,x(t,s), \DC x(t,s), 
^CD_{\gamma^1}^{\alpha_1(\cdot,\cdot),
\beta_1(\cdot,\cdot)}x(t,s),\right.\\
\left. \ldots, ^CD_{\gamma^n}^{\alpha_n(\cdot,\cdot),
\beta_n(\cdot,\cdot)}x(t,s), z(t) \right).
\end{multline*}
Consider the function
$$
\lambda(t)=\exp \left(-\int_a^t 
\int_{\Omega} \partial_{2n+4}
\left[x,z \right]_{n, \gamma}^{\alpha, \beta}(\tau,s) d^ns d\tau \right).
$$

\begin{theorem}
If $(x,z)$ is an extremizer to the functional \eqref{funct_siv1}, 
then $(x,z)$ satisfies the fractional differential equation
\begin{multline}
\label{CEL1_Herg2}
\partial_{n+2} L[x,z]_{n, \gamma}^{\alpha, \beta}(t,s)\lambda(t)
+D{_{\overline{\gamma}}^{\b,\a}}\left(\lambda(t)\partial_{n+3} 
L[x,z]_{n, \gamma}^{\alpha, \beta}(t,s)\right)\\
+ \sum_{i=1}^{n} D_{\overline{\gamma}^{i}}^{\beta_i(\cdot,\cdot),
\alpha_i(\cdot,\cdot)}\left(\lambda(t) \partial_{n+3+i}
L[x,z]_{n, \gamma}^{\alpha, \beta}(t,s)\right)=0
\end{multline}
on $[a,T] \times \Omega$ and
\begin{equation}
\label{CEL2_Herg2}
\gamma_2\left({\LDb} \left(\lambda(t)\partial_{n+3} 
L[x,z]_{n, \gamma}^{\alpha, \beta}(t,s)\right)
-{ _TD{_t^{\b}}\left(\lambda (t) \partial_{n+3} 
L[x,z]_{n, \gamma}^{\alpha, \beta}(t,s)\right)}\right)=0
\end{equation}
on $[T,b]\times \Omega$. Moreover, $(x,z)$ 
satisfies the transversality condition
\begin{multline}
\label{CTransH2}
\Bigl[\gamma_1 {_tI_T^{1-\a}} \left(\lambda (t) \partial_{n+3}
L[x,z]_{n, \gamma}^{\alpha, \beta}(t,s)\right) \\
-\gamma_2 {_TI_t^{1-\b}} \left(\lambda(t)\partial_{n+3}
L[x,z]_{n, \gamma}^{\alpha, \beta}(t,s)\right)\Bigr]_{t=T}=0,
\quad s \in\Omega. 
\end{multline}
If $T<b$, then $\displaystyle \int_{\Omega}
L[x,z]_{n, \gamma}^{\alpha, \beta}(T,s)d^{n}s=0$.
\end{theorem}

\begin{proof}
Let $x$ be a solution to the problem. Consider an admissible variation of $x$, 
$\overline {x}(t,s)= x(t,s)+\e{h(t,s)}$, where $h\in C^1(P)$ is an arbitrary 
perturbing curve and $\e \in\mathbb{R}$ is such that $|\e|\ll 1$. 
Consequently, from the boundary condition \eqref{herg:bound},
$h(t,s)=0$ for all $(t,s)\in \partial P$. On the other hand, 
consider an admissible variation of $z$, $\overline {z}= z+\e\theta$, 
where $\theta$ is a perturbing curve such that $\theta (a)=0$ and
$$
\theta(t)= \dfrac{d}{d \varepsilon} 
z(\overline {x},t) \biggr\rvert_{\varepsilon=0}.
$$
Differentiating $\theta(t)$ with respect to $t$, we obtain that
\begin{equation*}
\begin{split}
\dfrac{d}{dt}\theta (t)
&=\dfrac{d}{dt} \dfrac{d}{d\varepsilon} z(\overline {x},t) 
\biggr\rvert_{\varepsilon=0}\\
&= \dfrac{d}{d\varepsilon}\dfrac{d}{dt} z\left(\overline {x},t\right) 
\biggr\rvert_{\varepsilon=0}\\
&= \dfrac{d}{d\varepsilon} \int_{\Omega}
L[\overline{x},z]_{n, \gamma}^{\alpha, \beta}(t,s) d^{n}s 
\biggr\rvert_{\varepsilon=0}.
\end{split}
\end{equation*}
We conclude that
\begin{multline*}
\dot{\theta}(t)
=\int_{\Omega} \left( \partial_{n+2}
L[x,z]_{n, \gamma}^{\alpha, \beta}(t) h(t,s) 
+\partial_{n+3} 
L[x,z]_{n, \gamma}^{\alpha, \beta}(t,s) \DC h(t,s)\right.\\
+\left. \sum_{i=1}^{n}\partial_{n+3+i}
L[x,z]_{n, \gamma}^{\alpha, \beta}(t,s) ^CD_{\gamma^{i}}^{\alpha_i(\cdot,\cdot),
\beta_i(\cdot,\cdot)}h(t,s)+\partial_{2n+4}
L[x,z]_{n, \gamma}^{\alpha, \beta}(t,s) \theta(t)\right) d^{n}s.
\end{multline*}
To simplify the notation, define
$$ 
B(t)=\int_{\Omega}\partial_{2n+4}
L[x,z]_{n, \gamma}^{\alpha, \beta}(t,s)d^{n}s 
$$
and
\begin{multline*}
A(t)
=\int_{\Omega}\Bigl( \partial_{n+2}
L[x,z]_{n, \gamma}^{\alpha, \beta}(t) h(t,s) 
+\partial_{n+3} L[x,z]_{n, \gamma}^{\alpha, \beta}(t,s) 
\DC h(t,s)\\
+\sum_{i=1}^{n}\partial_{n+3+i}
L[x,z]_{n, \gamma}^{\alpha, \beta}(t,s) 
^CD_{\gamma^{i}}^{\alpha_i(\cdot,\cdot),
\beta_i(\cdot,\cdot)}h(t,s) \Bigr) d^{n}s.
\end{multline*}
Then, we obtain the linear differential equation
$$
\dot{\theta}(t)-B(t)\theta(t)=A(t),
$$
whose solution is
\begin{equation*}
\theta(T)\lambda(T) - \theta(a) = \int_a^{T} A(t) \lambda(t) dt.
\end{equation*}
Since $\theta(a)=\theta(T)=0$, we get
\begin{equation}
\label{solutionPH2}
\int_a^{T} A(t)\lambda(t)dt = 0.
\end{equation}
Considering only the second term in \eqref{solutionPH2}, we can write
\begin{equation*}
\begin{split}
\int_a^{T} 
&\int_{\Omega}\lambda(t) \partial_{n+3} 
L[x,z]_{n, \gamma}^{\alpha, \beta}(t,s) 
\left( \gamma_{1} \LC h(t,s) + \gamma_{2} \RCb h(t,s) \right) d^{n}s dt\\
&=\gamma_{1} \int_a^{T} \int_{\Omega}\lambda(t) \partial_{n+3} 
L[x,z]_{n, \gamma}^{\alpha, \beta}(t,s) \LC h(t,s)d^{n}s dt\\
&\quad +\gamma_{2}  \left[ \int_a^{b} \int_{\Omega}\lambda(t) \partial_{n+3} 
L[x,z]_{n, \gamma}^{\alpha, \beta}(t,s) \RCb h(t,s) d^{n}s dt\right.\\ 
& \qquad\qquad - \left.\int_T^{b} \int_{\Omega}\lambda(t) 
\partial_{n+3} L[x,z]_{n, \gamma}^{\alpha, \beta}(t,s) \RCb h(t,s) d^{n}sdt \right].
\end{split}
\end{equation*}
Let $\overline {\gamma} =(\gamma_{_{2}}, \gamma_{1})$. 
Integrating by parts (cf. Theorem~\ref{thm:FIP}) 
and since $h(a,s)=0$ and $h(b,s)=0$ for all $s \in \Omega$, 
we obtain the following expression:
\begin{equation*}
\begin{split}
&\int_a^{T} \int_{\Omega} h(t,s) D_{ \overline{\gamma}}^{\b,\a} 
\left( \lambda(t) \partial_{n+3} 
L[x,z]_{n, \gamma}^{\alpha, \beta}(t,s) \right) d^{n}sdt \\
&+ \gamma_{2}\int_T^{b} \int_{\Omega} h(t,s) \left[ 
\LDb  \left( \lambda(t) \partial_{n+3} 
L[x,z]_{n, \gamma}^{\alpha, \beta}(t,s) \right)\right.\\  
&\qquad\qquad\qquad\qquad\qquad 
-\left. _TD_t^{\b} \left( \lambda(t) \partial_{n+3} 
L[x,z]_{n, \gamma}^{\alpha, \beta}(t,s) \right) \right] d^{n}sdt\\
& + \int_{\Omega} h(T,s) \left[ \gamma_{1} {}_tI_{T}^{1-\a} \left( \lambda(t) 
\partial_{n+3} L[x,z]_{n, \gamma}^{\alpha, \beta}(t,s) \right)\right.\\
&\qquad\qquad\qquad\qquad  \left. - \gamma_{2} \, _TI_{t}^{1-\b} 
\left( \lambda(t) \partial_{n+3} 
L[x,z]_{n, \gamma}^{\alpha, \beta}(t,s) \right) d^n s\right]_{t=T}.
\end{split}
\end{equation*}
Doing similarly for the $(i+2)$th term of \eqref{solutionPH2}, 
$i=1,\ldots, n$, letting $\overline {\gamma}^i =(\gamma_{_{2}}^i, 
\gamma_{1}^{i})$, and since $h(t,a_i)=h(t,b_i)=0$ for all $t \in [a,b]$, 
we obtain
\begin{multline*}
\int_a^{T} \int_{\Omega}\lambda(t) \partial_{n+3+i} 
L[x,z]_{n, \gamma}^{\alpha, \beta}(t,s) \left( \gamma_{1}^{i} 
{_{a_i}^CD_{s_i}^{\alpha_{i}(\cdot,\cdot)} }h(t,s) 
+ \gamma_{2}^{i} {_{s_i}^CD_{b_i}^{\beta_{i}(\cdot,\cdot)}} h(t,s)\right) d^{n}s dt\\
=\int_a^{T} \int_{\Omega} h(t,s) D_{\overline {\gamma}^i}^{~\beta_{i}(\cdot,\cdot),
\alpha_{i}(\cdot,\cdot)} \left(\lambda(t)\partial_{n+3+i} 
L[x,z]_{n, \gamma}^{\alpha, \beta}(t,s) \right) d^{n}s dt.
\end{multline*}
Substituting these relations into \eqref{solutionPH2}, we deduce that
\begin{equation*}
\begin{split}
\int_a^{T} &\int_{\Omega} h(t,s)\left[ \partial_{n+2}
L[x,z]_{n, \gamma}^{\alpha, \beta}(t,s)\lambda(t) 
+ D_{ \overline{\gamma}}^{\b,\a} 
\left( \lambda(t) \partial_{n+3} 
L[x,z]_{n, \gamma}^{\alpha, \beta}(t,s) \right) \right.\\
&+\sum_{i=1}^{n} D_{\overline {\gamma}^i}^{~\beta_{i}(\cdot,\cdot),
\alpha_{i}(\cdot,\cdot)} \left(\lambda(t)
\partial_{n+3+i} L[x,z]_{n, \gamma}^{\alpha, \beta}(t,s) \right) d^{n}sdt \\
&+ \gamma_{2}\int_T^{b} \int_{\Omega} h(t,s) \left[ 
\LDb\left( \lambda(t) \partial_{n+3} 
L[x,z]_{n, \gamma}^{\alpha, \beta}(t,s) \right)\right.\\
& \qquad\qquad\qquad\qquad\quad \left.
- _TD_t^{\b} \left( \lambda(t) \partial_{n+3} 
L[x,z]_{n, \gamma}^{\alpha, \beta}(t,s) \right) \right] d^{n}sdt\\
&+ \int_{\Omega} h(T,s) \left[ \gamma_{1} {}_tI_{T}^{1-\a} 
\left( \lambda(t) \partial_{n+3} 
L[x,z]_{n, \gamma}^{\alpha, \beta}(t,s) \right)\right.\\
&\qquad\qquad\qquad\qquad\left. 
-\gamma_{2} \, _TI_{t}^{1-\b} \left( \lambda(t) 
\partial_{n+3} L[x,z]_{n, \gamma}^{\alpha, \beta}(t,s) 
\right)d^ns \right]_{t=T}.
\end{split}
\end{equation*}
We get the Euler--Lagrange equations 
\eqref{CEL1_Herg2}--\eqref{CEL2_Herg2} 
and the transversality condition \eqref{CTransH2}
with appropriate choices of $h$.
\end{proof}


\section{Illustrative examples}

We present three examples.

\begin{example}
\label{ex:2}
Consider
\begin{equation}
\label{exemp2}
\begin{gathered}
\dot{z}(t)=\left(\DC x(t)\right)^2+z(t)+t^2-1, \quad t\in [0,3],\\
x(0)=1, \quad z(0)=0.
\end{gathered}
\end{equation}
In this case, $\lambda(t)=\exp(-t)$.  
The necessary optimality conditions \eqref{CEL1_Herg}--\eqref{CEL2_Herg} 
of Theorem~\ref{mainteo} hold for $\overline{x}(t) \equiv 1$. 
If we replace $x$ by $\overline{x}$ in \eqref{exemp2}, we obtain
\begin{gather*}
\dot{z}(t)-z(t)=t^2-1, \quad t\in [0,3],\\
z(0)=0,
\end{gather*}
whose solution is
\begin{equation}
\label{z:sol}
z(t)=\exp(t)-(t+1)^2.
\end{equation}
The last transversality condition of Theorem~\ref{mainteo} asserts that
$$
L[\overline x,z]_\gamma^{\alpha, \beta}(T)=0 \Leftrightarrow \exp(T)-2T-2=0,
$$
whose solution is approximately
$$
T\approx 1.67835.
$$
We remark that $z$ \eqref{z:sol} actually attains a minimum value 
at this point (see Figure~\ref{heg1},~(a)):
$$
z(1.67835)\approx -1.81685.
$$
\end{example}

\begin{example}
\label{ex:3}
Consider now
\begin{equation}
\label{exemp3}
\begin{gathered}
\dot{z}(t)=(t-1)\left(x^2(t)+z^2(t)+1\right), \quad t\in [0,3],\\
x(0)=0, \quad z(0)=0.
\end{gathered}
\end{equation}
Since the first Euler--Lagrange equation \eqref{CEL1_Herg} reads
$$
(t-1)x(t)=0 \quad \forall \, t\in[0,T],
$$
we see that $\overline x(t) \equiv 0$ is a solution of this equation. 
The second transversality condition of \eqref{CTransH} 
asserts that, at $t=T$, we must have
$$
L[\overline x,z]_\gamma^{\alpha, \beta}(t)=0,
$$
that is,
$$
(t-1)(z^2(t)+1)=0,
$$
and so $T=1$ is a solution of this equation.
Substituting $x$ by $\overline{x}$ in \eqref{exemp3}, we get
\begin{gather*}
\dot{z}(t)=(t-1)(z^2(t)+1), \quad t\in [0,3],\\
z(0)=0.
\end{gather*}
The solution to this Cauchy problem 
is the function 
$$
z(t)=\tan\left(\frac{t^2}{2}-t\right)
$$
(see Figure~\ref{heg1},~(b)) and the minimum value is
$$
z(1)=\tan\left(-\frac{1}{2}\right).
$$
\end{example}

\begin{example}
\label{ex:1}
For our last example, consider
\begin{equation}
\label{exemp}
\begin{gathered}
\dot{z}(t)=\left(\DC x(t) - f(t)\right)^2+t^2-1, \quad t\in [0,3],\\
x(0)=0, \quad z(0)=0,
\end{gathered}
\end{equation}
where
$$
f(t) := \frac{t^{1-\alpha(t)}}{2\Gamma(2-\alpha(t))}
-\frac{(3-t)^{1-\beta(t)}}{2\Gamma(2-\beta(t))}.
$$
In this case, $\lambda(t)\equiv 1$. We intend to find a pair $(x,z)$, 
satisfying all the conditions in \eqref{exemp}, 
for which $z(T)$ attains a minimum value.
It is easy to verify that $\overline{x}(t) = t$ and $T=1$ satisfy 
the necessary conditions given by Theorem~\ref{mainteo}. 
Replacing $x$ by $\overline{x}$ in system \eqref{exemp}, 
we get a Cauchy problem of form
\begin{gather*}
\dot{z}(t)=t^2-1, \quad t\in [0,3],\\
z(0)=0,
\end{gather*}
whose solution is
$$
z(t)=\frac{t^3}{3}-t.
$$
Observe that this function attains a minimum value 
at $T=1$, which is $z(1)=-2/3$ (Figure~\ref{heg1},~(c)).
\end{example}

\begin{figure}[ht!]
\begin{center}
\subfigure[Extremal $z$ of Example~\ref{ex:2}.]{\includegraphics[scale=0.23]{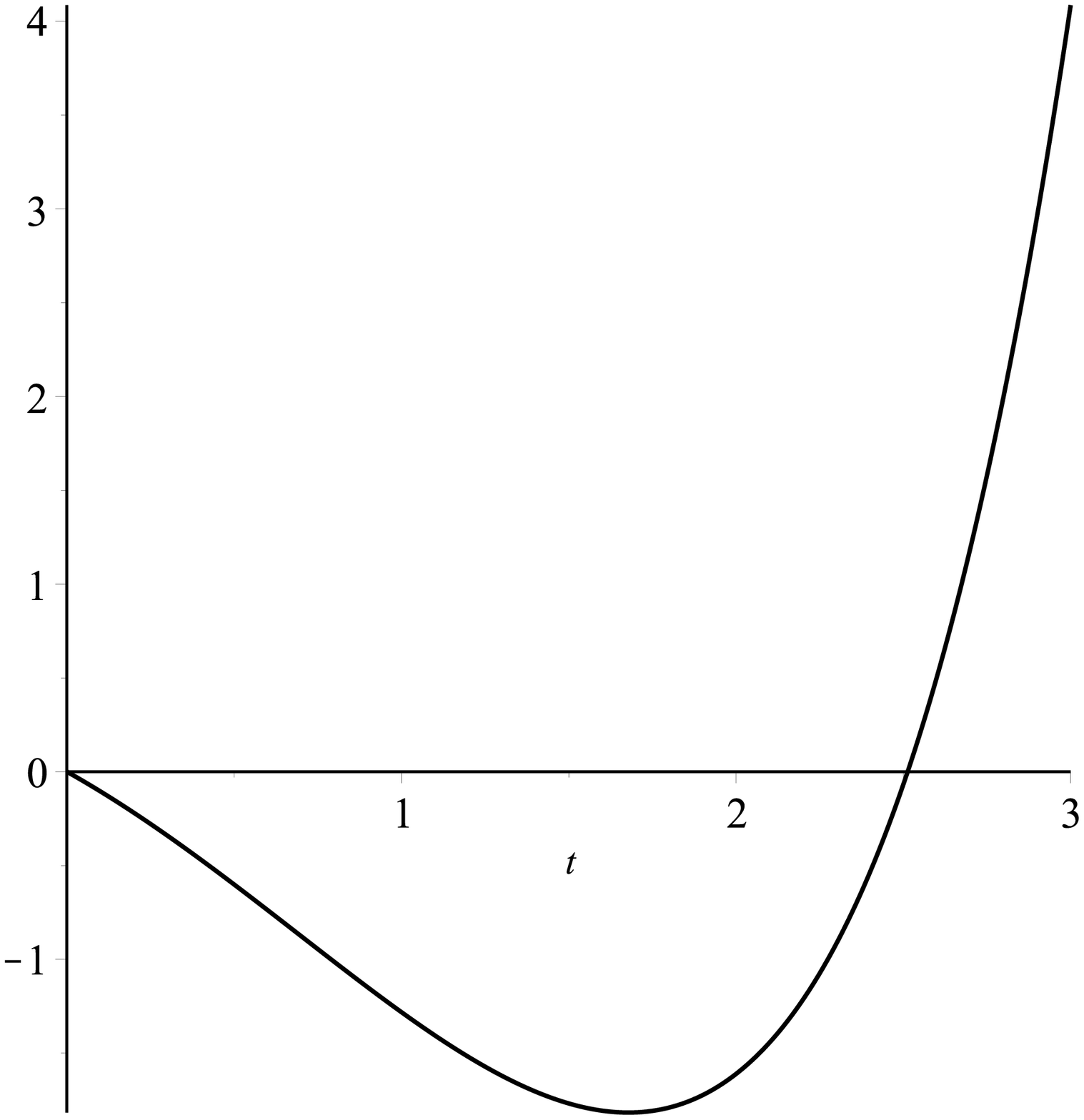}}
\subfigure[Extremal $z$ of Example~\ref{ex:3}.]{\includegraphics[scale=0.23]{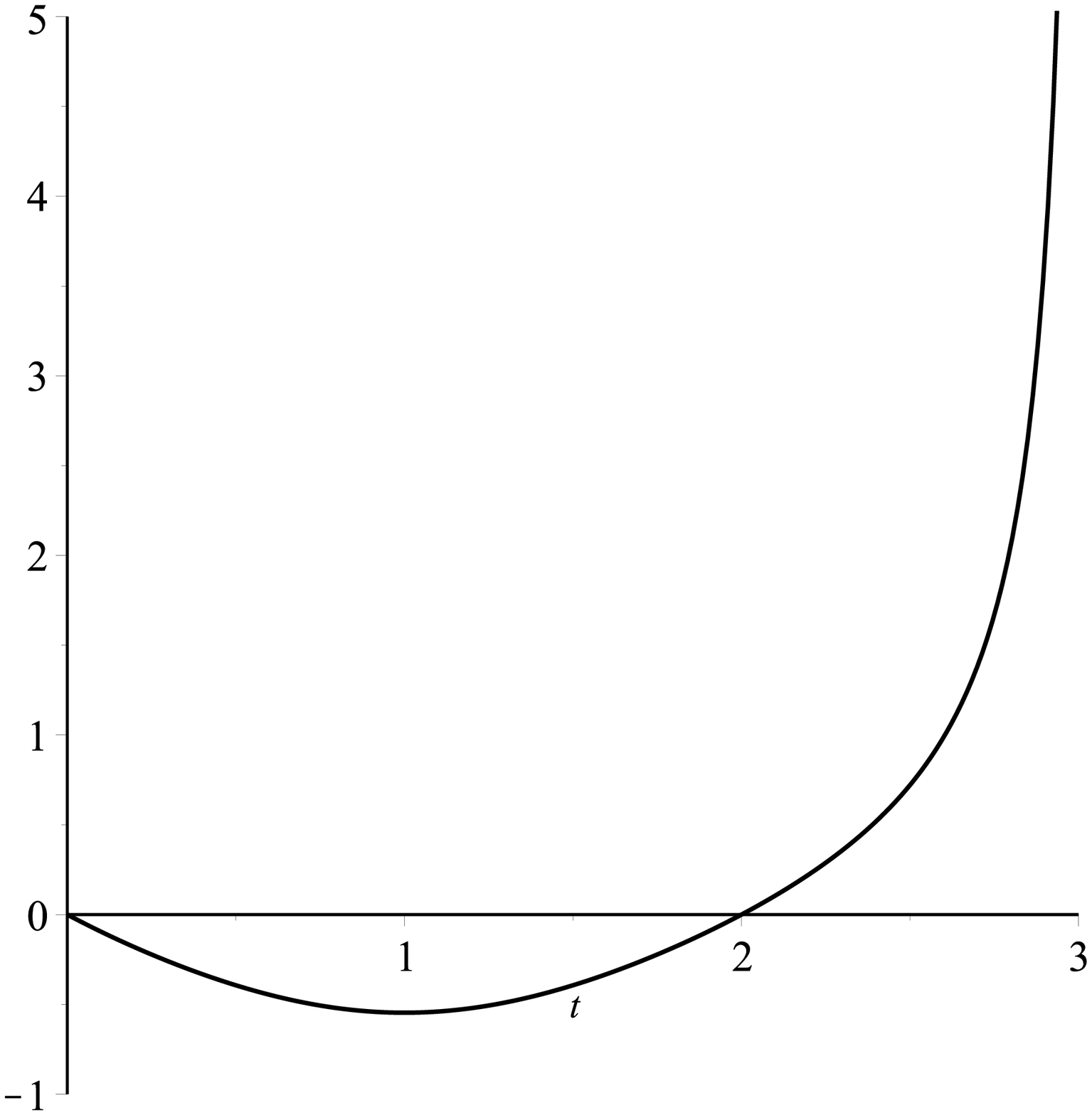}}
\subfigure[Extremal $z$ of Example~\ref{ex:1}.]{\includegraphics[scale=0.23]{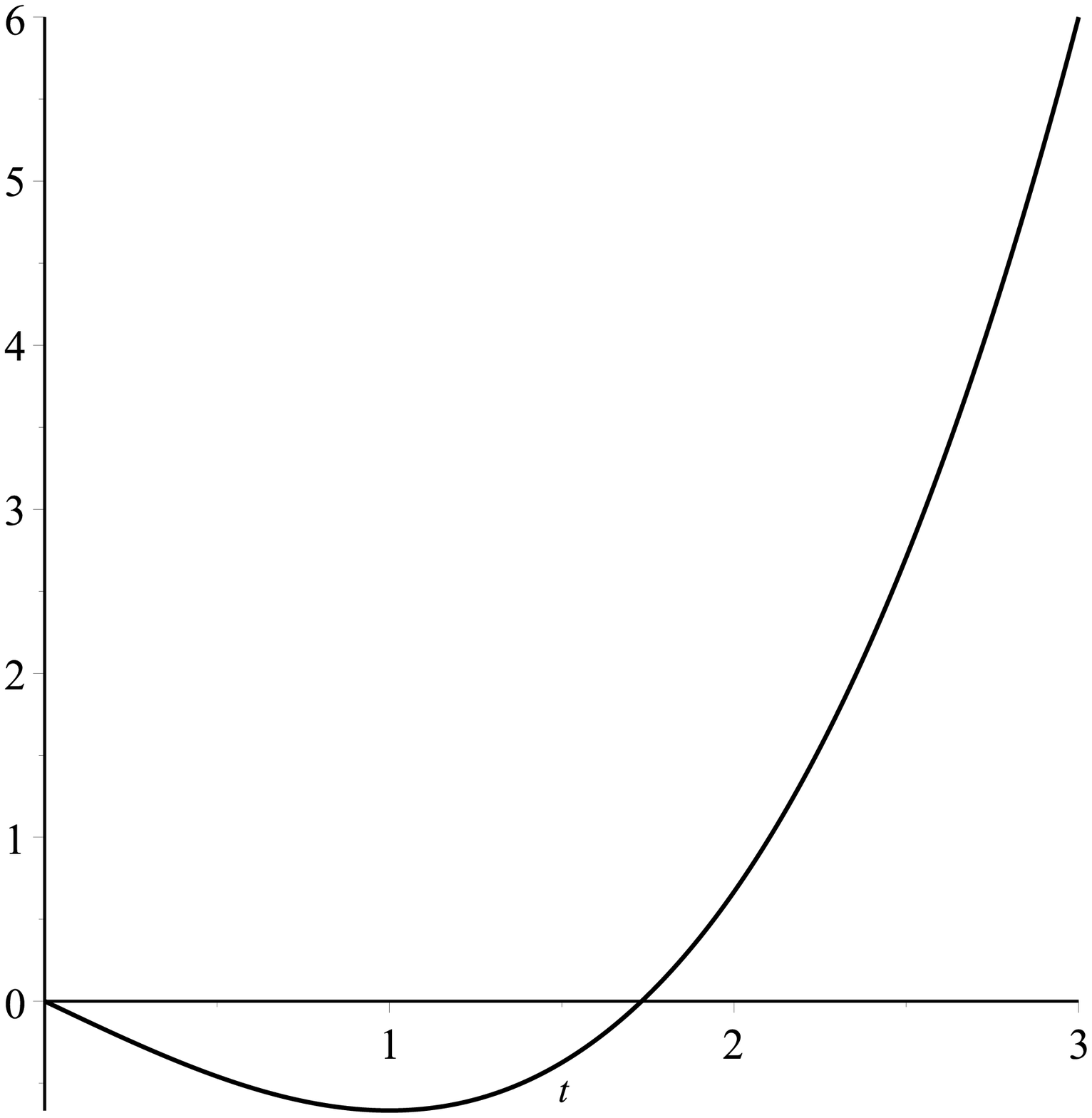}}
\caption{Graphics of function $z(\overline x,t)$.}\label{heg1}
\end{center}
\end{figure}


\section*{Acknowledgements}

The authors are grateful to two anonymous referees 
for their comments.



\medskip

Received July 29, 2016; Revised Feb 03, 2017; Accepted March 27, 2017.

\medskip



\begin{thebibliography}{99}

\bibitem{Almeida} (MR3274999) [10.3934/dcdsb.2014.19.2367]
\newblock R. Almeida and A. B. Malinowska,
\newblock Fractional variational principle of Herglotz,
\newblock \emph{Discrete Contin. Dyn. Syst. Ser. B} 
\textbf{19} (2014), no.~8, 2367--2381.
\newblock {\tt arXiv:1406.0717}

\bibitem{Coimbra}
\newblock F. M. Coimbra, C. M. Soon and M. H. Kobayashi,
\newblock The variable viscoelasticity operator,
\newblock \emph{Annalender Physik} \textbf{14} (2005), 378--389.

\bibitem{Georgieva} (MR1962221) 
\newblock B. Georgieva and R. Guenther,
\newblock First Noether-type theorem for the generalized variational principle of Herglotz,
\newblock \emph{Topol. Methods Nonlinear Anal.} \textbf{20} (2002), no.~2, 261--273.

\bibitem{Georgieva:sev} (MR2003940) [10.1063/1.1597419]
\newblock B. Georgieva, R. Guenther and T. Bodurov,
\newblock Generalized variational principle of Herglotz for several independent variables,
\newblock \emph{J. Math. Phys.} \textbf{44} (2003), no.~9, 3911--3927.

\bibitem{Guenther} (MR1391230) [10.1137/1038042]
\newblock R. B. Guenther, J. A. Gottsch and D. B. Kramer,
\newblock The Herglotz algorithm for constructing canonical transformations,
\newblock \emph{SIAM Rev.} \textbf{38} (1996), no.~2, 287--293.

\bibitem{Guenther:book}
\newblock R. B. Guenther, C. M. Guenther and J. A. Gottsch,
\newblock The Herglotz Lectures on Contact Transformations and Hamiltonian Systems,
\newblock \emph{Lecture Notes in Nonlinear Analysis}, Vol.~1, 
Juliusz Schauder Center for Nonlinear Studies, 
Nicholas Copernicus University, Tor\'un, 1996.

\bibitem{Herglotz}
\newblock G. Herglotz,
\newblock Ber\"uhrungstransformationen, 
Lectures at the University of G\"ottingen, 
G\"ottingen, 1930.

\bibitem{Malin:Tor} (MR2846374) [10.2478/s13540-011-0032-6]
\newblock A. B. Malinowska and D. F. M. Torres,
\newblock Fractional calculus of variations for a combined Caputo derivative,
\newblock \emph{Fract. Calc. Appl. Anal.} \textbf{14} (2011), 523--537.
\newblock {\tt arXiv:1109.4664}

\bibitem{MyID:207} (MR2861352) [10.1016/j.na.2011.01.010]
\newblock T. Odzijewicz, A. B. Malinowska and D. F. M. Torres,
\newblock Fractional variational calculus with classical and combined Caputo derivatives,
\newblock \emph{Nonlinear Anal.} \textbf{75} (2012), 1507--1515.
\newblock {\tt arXiv:1101.2932}

\bibitem{Tatiana:IDOTA2011} (MR3060420) [10.1007/978-3-0348-0516-2_16]
\newblock T. Odzijewicz, A. B. Malinowska and D. F. M. Torres,
\newblock Fractional variational calculus of variable order.
\newblock In: \emph{Advances in harmonic analysis and operator theory}, 291--301,
\newblock Oper. Theory Adv. Appl., 229, Birkh\"auser/Springer Basel AG, Basel, 2013.
\newblock {\tt arXiv:1110.4141}

\bibitem{Od} [10.2478/s11534-013-0208-2]
\newblock T. Odzijewicz, A. B. Malinowska and D. F. M. Torres,
\newblock Noether's theorem for fractional variational problems of variable order,
\newblock \emph{Cent. Eur. J. Phys.} \textbf{11} (2013), 691--701.
\newblock {\tt arXiv:1303.4075}

\bibitem{Samko_2} (MR1421643) [10.1080/10652469308819027]
\newblock S. G. Samko and B. Ross,
\newblock Integration and differentiation to a variable fractional order,
\newblock \emph{Integral Transform. Spec. Funct.} \textbf{1} (1993), no.~4, 277--300.

\bibitem{Santos:Viet} (MR3286693) [10.1007/s10013-013-0048-9]
\newblock S. P. S Santos, N. Martins and D. F. M. Torres,
\newblock Higher-order variational problems of Herglotz type,
\newblock \emph{Vietnam J. Math.} \textbf{42} (2014), no.~4, 409--419.
\newblock {\tt arXiv:1309.6518}

\bibitem{Santos:Disc} (MR3392640) [10.3934/dcds.2015.35.4593]
\newblock S. P. S Santos, N. Martins and D. F. M Torres,
\newblock Variational problems of Herglotz type with time delay:  
DuBois-Reymond condition and Noether's first theorem,
\newblock \emph{Discrete Contin. Dyn. Syst.} \textbf{35} (2015), no.~9, 4593--4610.
\newblock {\tt arXiv:1501.04873}

\bibitem{Santos:Spri} [10.1007/978-3-319-20352-2_7]
\newblock S. P. S. Santos, N. Martins and D. F. M. Torres,
\newblock An optimal control approach to Herglotz variational problems.
\newblock In: \emph{Optimization in the Natural Sciences} 
(eds. A. Plakhov, T. Tchemisova and A. Freitas), 
Communications in Computer and Information Science, 
Vol.~499, Springer, 2015, 107--117.
\newblock {\tt arXiv:1412.0433}

\bibitem{MR3462534} (MR3462534) [10.3934/proc.2015.990]
\newblock S. P. S. Santos, N. Martins and D. F. M. Torres, 
\newblock Noether's theorem for higher-order variational 
problems of Herglotz type, 
\newblock \emph{Discrete Contin. Dyn. Syst.} {\bf 2015} (2015), 
Dynamical systems, differential equations and applications. 
10th AIMS Conference. Suppl., 990--999. 
\newblock {\tt arXiv:1507.05911}

\bibitem{MyID:342}
\newblock S. P. S. Santos, N. Martins and D. F. M. Torres, 
\newblock Higher-order variational problems of Herglotz type with time delay,
\newblock \emph{Pure and Applied Functional Analysis} {\bf 1} (2016), no.~2, 291--307.
\newblock {\tt arXiv:1603.04034}

\bibitem{Sun} [10.1088/0256-307X/30/4/046601]
\newblock H. Sun, S. Hu, Y. Chen, W. Chen and Z. Yu,
\newblock A dynamic-order fractional dynamic system,
\newblock \emph{Chinese Phys. Lett.} \textbf{30} (2013), 046601, 4~pp.

\bibitem{Tavares} (MR3325357) [10.1080/02331934.2015.1010088]
\newblock D. Tavares, R. Almeida and D. F. M. Torres,
\newblock Optimality conditions for fractional variational 
problems with dependence on a combined caputo derivative of variable order,
\newblock \emph{Optim.} \textbf{64} (2015), 1381--1391.
\newblock {\tt arXiv:1501.02082}

\end{thebibliography}
\end{document}